\documentclass[12pt]{article}
\usepackage[utf8]{inputenc}
\usepackage[english]{babel}
\usepackage{amsthm}
\usepackage{amsmath}
\usepackage{amsfonts}
\usepackage{mathtools}
\usepackage{enumitem}
\usepackage{xcolor}
\usepackage{stmaryrd}
\usepackage{amssymb}
\usepackage{mathrsfs}

\setlist[enumerate,1]{label = (\roman*)}

\theoremstyle{definition}
\newtheorem{definition}{Definition}[section]

\newtheorem{theorem}[definition]{Theorem}
\newtheorem{lemma}[definition]{Lemma}

\newtheorem{fact}[definition]{Fact}

\newcommand{\wins}{\mathbin{\raisebox{0.8\depth}{$\uparrow$}}}

\newcommand{\playerone}{\mathrm{\mathbf{I}}}
\newcommand{\playertwo}{\mathrm{\mathbf{II}}}
\DeclareMathOperator{\EF}{EF}
\newcommand{\EFdep}{\EF^{\mathrm{dep}}}
\newcommand{\winsu}{\wins_{\!\textrm u}}
\newcommand{\W}{\mathcal W}
\newcommand{\U}{\mathcal U}

\newcommand{\A}{\mathfrak{A}}
\newcommand{\B}{\mathfrak{B}}

\DeclareMathOperator{\qr}{qr}
\DeclareMathOperator{\ar}{ar}

\newcommand{\dep}{\mathop{=}}

\newcommand{\posEquiv}{\mathbin{\Rrightarrow}}

\newcommand{\deplog}{\mathcal D}

\DeclareMathOperator{\Pow}{\mathcal{P}}
\DeclareMathOperator{\dom}{dom}

\def\ma{\A}
\def\mb{\B}
\title{A New Ehrenfeucht--Fraïssé Game for Dependence Logic}
\author{Joni Puljujärvi\thanks{The first author was supported by U.K. EPSRC Research Fellowship EP/V040944/1, Resources in Computation.}\\University College London\\ London, UK \and  Jouko Väänänen\thanks{The second author has received funding from the Academy of Finland (decision number 368671) and the European Research Council (ERC) under the European Union’s Horizon 2020 research and innovation programme (grant agreement No 101020762).}\\ University of Helsinki\\ Helsinki, Finland}
\begin{document}
\maketitle

\begin{abstract}
    We define a new Ehrenfeucht--Fraïssé game for dependence logic. The previously known rendition of such a game  was based on moves that are teams. Since teams can be massive, making team moves  may be quite complicated. To remedy this, our new Ehrenfeucht--Fraïssé game for dependence logic has only moves that consist of single elements, as in the classical Ehrenfeucht--Fraïssé game of first order logic. A new feature of the game is that a player can declare that their move is made on the basis of certain previous moves only and thereby in a sense  independent of other moves. We show that our game characterizes elementary equivalence in dependence logic.
\end{abstract}

\section{Introduction}

The Ehrenfeucht--Fraïssé game \cite{MR126370} is a fundamental tool of first order logic. We define the corresponding game for dependence logic \cite{MR2351449}---an extension of first order logic by \emph{dependence atoms} $\dep(x,y)$, with the intuitive meaning ``$y$ is completely determined by $x$"---and show that the game characterizes elementary equivalence in this logic.

Previously, versions of the Ehrenfeucht--Fraïssé game have been defined for a variety of extensions and fragments of first order logic, in particular infinitary logics (\cite{MR209132}, \cite{MR252208}, \cite{MR292656}), second order logic (\cite{MR191791}), logics with generalized quantifiers (\cite{MR314591}, \cite{MR427043}, \cite{MR462912}, \cite{MR597058}, \cite{MR1336413}) and logics with a fixed finite number of variables (\cite{MR685362}, \cite{MR666822}, \cite{MR1167033}). Moreover, one version of the  Ehrenfeucht--Fraïssé game for the very dependence logic that is our subject in this paper  was already defined in \cite{MR2351449}. However, that game was ``second order" in the sense that moves were sets of assignments, or teams, as they are called. Similarly, there is a ``second-order" game for inclusion logic \cite{MR3111746}, as well as one for inquisitive logic \cite{GRILLETTI_CIARDELLI_2023}. Our new game is ``first order" in the sense that moves are elements as in the original Ehrenfeucht--Fraïssé game of \cite{MR126370}. We show that our game captures dependence logic in the same way as the original Ehrenfeucht--Fraïssé game captures first order logic.

The ordinary Ehrenfeucht--Fraïssé game is a tool for comparing two models. The game is a perfect-information zero-sum game with two players, $\playerone$ and $\playertwo$. Winning strategies of player $\playertwo$ are witnesses to levels of similarity (or isomorphism) of the models, manifested by elementary equivalence up to a quantifier rank. Winning strategies of player $\playerone$ are witnesses to a difference between the models, manifested by a first-order sentence which is true in one model but not in the other. It is possible to think of the game as a syntax-free approach to comparing the models. We can use the game to express inability to separate the models with a first-order sentence of a certain size. Equivalently, we can use the game to say that one of the models has a first-order property that the other model does not have. In this analysis, the concept of ``first order'' can be varied by modifying the game.

In using the Ehrenfeucht--Fraïssé game approach for the logical analysis of dependence, something we aim to do in this paper, we have two models and we want to analyze whether they have the same dependence-type properties. For example, there may be a  definable binary relation which in one model contains a one-one function but in the other model perhaps not. Or there may be two definable unary relations $P$ and $Q$ such that in one model $|P|\le |Q|$ and in the other model perhaps $|P|>|Q|$. Or there may be a definable linear order which is a well-order in one model but perhaps not in the other. Or there may be a definable graph relation which is 3-colourable in one model but perhaps not in the other. How does the Ehrenfeucht--Fraïssé game capture such differences?

A possible solution is to use a second order Ehrenfeucht--Fraïssé game in which the players play relations rather than elements, or teams as in \cite{MR2351449}. An interesting case is the game for the Henkin (i.e. partially ordered) quantifier, where moves involve choosing a function which is then thrown away \cite{MR427043}. So positions in that game are first order even if moves are second order. This is relevant from the point of view of dependence logic because sentences of the latter have the same expressive power as sentences which start with a partially ordered quantifier prefix followed by a first order formula. However, in this paper we want to stay on the first-order level and focus on an Ehrenfeucht--Fraïssé game in which the moves are elements of the models, not teams or functions. Admittedly, it is not absolutely clear that there is a real difference, as we impose second order conditions on winning strategies.

One important feature of dependence logic that comes to play a role is the lack of negation, except in front of first-order atomic formulas. Therefore, we aim at defining a variant of the Ehrenfeucht--Fraïssé game for $\ma$ and $\mb$ such that if player $\playertwo$ has a winning strategy, then $\ma\models\phi$ implies $\mb\models\phi$ for all dependence logic sentences $\phi$. By stating the same also in converse order, i.e. existence of a winning strategy for $\playertwo$ in the game for $\mb$ and $\ma$, we obtain a criterion for elementary equivalence in the usual sense.

Let us consider, as an example,  two graphs $\ma$ and $\mb$. Let us assume that $\ma$ is 2-colourable but $\mb$ is  not. We assume the possibility to refer to two colours in both graphs and the question is whether the vertices can be coloured with these colours in a way that is required by the concept of 2-colouring.\footnote{One way to do this is to add a constant symbol for each colour. Another option is to simply play two additional rounds of the game in the beginning and treat the two first moves of each player as the two colours in each graph. Then a colouring of either of the graphs will correspond to a function on the graph that maps each element to one of these two distinguished elements.} Thus $\ma$ has the dependence-type property that we can assign one of two colours to each vertex in such a way that neighboring vertices get different colour. How would player $\playerone$ take advantage of this colouring $f$ in $\ma$ and the lack of any 2-colouring in $\mb$? An obvious idea is the following:

\begin{enumerate}[label = (\arabic*)]
    \item $\playerone$ picks $x_0$ in $\mb$ asking $\playertwo$ to give its colour.
    \item $\playertwo$ retaliates by picking $y_0$ in $\ma$ and asking $\playerone$ to give first its colour.
    \item $\playerone$ gives a colour $x_1$ for $y_0$, making the commitment that he chose the colour {\bf on the basis of knowing $y_0$ only}.

    \item Now $\playertwo$ gives some colour $y_1$ for $x_0$, making the commitment that she chose the colour {\bf on the basis of knowing $x_0$ only}.

    \item $\playerone$ picks again some $x_2$ in $\mb$ asking again $\playertwo$ to give its colour.
    \item $\playertwo$ again retaliates by picking $y_2$ in $\ma$ and asking $\playerone$ to give its colour.
    \item $\playerone$ gives a colour $x_3$ for $y_2$, claiming that he chose the colour {\bf on the basis of knowing $y_2$ only}.
    \item $\playertwo$ gives some colour $y_3$ for $x_2$, claiming that she chose the colour {\bf on the basis of knowing $x_2$ only}.

    \item\label{item: winning condition}
    Now $\playertwo$ has won if
    \begin{enumerate}
        \item $x_0=x_2\implies y_0=y_2$,
        \item $x_0Ex_2\implies y_0Ey_2$, and
        \item $x_1\ne x_3 \implies y_1\ne y_3$.
    \end{enumerate}
\end{enumerate}

A pair $\{s_0,s_1\}$ of plays of the above game is said to be \emph{$\playerone$-good} if it satisfies $\dep(y_0,x_1)\wedge\dep(y_2,x_3)$ and \emph{$\playertwo$-good} if it satisfies $\dep(x_0,y_1)\wedge\dep(x_2,y_3)$. A strategy $\tau$ of $\playertwo$ is called a \emph{uniform winning strategy} if every $\playerone$-good pair of plays in which $\playertwo$ has used the strategy  $\tau$ is $\playertwo$-good  and satisfies condition \ref{item: winning condition} above. 

We show that $\playertwo$ cannot have a uniform winning strategy in this game, assuming the players honour their commitments. Suppose she has and let us call it $\tau$. Let us  play so that $\playerone$ plays $x_0=x_2$ but otherwise $x_0$ is arbitrarily. Then she  plays some $y_1$. Let us denote this $y_1=g(x_0)$. The function $g$ cannot be a 2-colouring of $\mb$ because $\mb$ is not 2-colourable. Thus there are $b_0$ and $b_1$ in $\mb$ such that $b_0E^\mb b_1$ but $g(b_0)=g(b_1)$. Let us consider two plays $s_0$ and $s_1$ (see Table~\ref{tab:placeholder}):

\begin{table}[h]
\begin{center}
    \begin{tabular}{|c|c|cc|}
     \hline
        \multicolumn{2}{|c}{} & $s_0$ & $s_1$ \\
        \hline
        $\playerone$&$ x_0$  & $b_0$ & $b_1$ \\
        $\playertwo$&$y_0=\tau(x_0)$ &  $a_0$ & $a_0'$ \\
        $\playerone$&$x_1$ &  $f(a_0)$ & $f(a_0')$ \\
        $\playertwo$&$y_1$  & $g(b_0)$ &$g(b_1)$\\
        $\playerone$&$x_2$ & $b_0$& $b_0$ \\
        $\playertwo$&$y_2$ & $a_1$  & $a_1'$ \\
        $\playerone$&$x_3$ & $f(a_1)$  & $f(a_1')$ \\
        $\playertwo$&$y_3$ &   $d_1$& $d_1'$\\
         \hline
    \end{tabular}
\end{center} \caption{Two plays.}
    \label{tab:placeholder}\end{table}

The pair $s_0,s_1$ is $\playerone$-good since $\playerone$ is using the 2-colouring $f$ in $A$. Since $\playertwo$ is using  $\tau$, the pair is also $\playertwo$-good. Hence $d_1=d'_1$. By condition \ref{item: winning condition} applied to $s_0$, we have $g(b_0)=d_1$. Condition~\ref{item: winning condition} applied to $s_1$ yields $a_0'E^\mb a_1'$, whence $f(a_0')\ne f(a_1')$, and by condition \ref{item: winning condition} again, still applied to $s_1$, $g(b_1)\ne d_1'$. But
\[
    g(b_1)=g(b_0)=d_1=d_1'.
\]
This contradiction shows that $\playertwo$ cannot have a uniform winning strategy.

 The above game is a prototype of the new Ehrenfeucht--Fraïssé game $\EFdep_n(\A,\B)$ we propose in Section~\ref{The new game}, after some preliminaries in Section~\ref{pre}. We also introduce the crucially important concept of a uniform winning strategy of $\playertwo$ in $\EFdep_n(\A,\B)$.
In Section~\ref{An Auxiliary Game},  an auxiliary game, a slight modification $\EF^\deplog_n((\A,X),(\B,Y))$ of the Ehrenfeucht--Fraïssé game presented in \cite{MR2351449}, is introduced. It is then shown, in Section~\ref{Auxiliary game vs. uniform game}, that if $\playertwo$ has a winning strategy in the auxiliary game $\EF^\deplog_n((\A,X),(\B,Y))$, she has a uniform winning strategy in  $\EFdep_n(\A,\B)$.
 Section~\ref{Uniform game vs. logic} is devoted to the main results of this paper: an equivalence of having a uniform winning strategy in $\EFdep_n(\A,\B)$ and preservation of truth from $\A$ to $\B$.

\section{Preliminaries}\label{pre}

Let $L$ be a vocabulary. Given an $L$-structure $\A$ and a set $D$ of variables, an \emph{assignment} of $\A$ with domain $D$ is a function $D\to A$. Given an element $a\in A$ and an assignment $s\in A^D$ and a variable $x$ (not necessarily in $D$), we we denote by $s[a/x]$ the assignment $s'$ of $\A$ with domain $D\cup\{x\}$, defined by
\[
    s'(y) =
    \begin{cases}
        a & \text{if $y = x$,} \\
        s(y) & \text{for $y\in D\setminus\{x\}$.}
    \end{cases}
\]
Given an assignment $s$ and a tuple $\vec x = (x_0,\dots,x_{n-1})\in\dom(s)^n$ of variables, we write $s(\vec x)$ as a shorthand for the tuple $(s(x_0),\dots,s(x_{n-1}))$.

A structure $\A$ and an assignment $s$ of $\A$ satisfying a first-order formula $\phi$, in symbols $\A\models_s\phi$, is defined as usual: $\A\models_s R(\vec x)$ if $s(\vec x)\in R^\A$; $\A\models_s x = y$ if $s(x) = s(y)$; $\A\models_s\neg\phi$ if $\A\not\models_s\phi$; $\A\models_s \phi\land\psi$ if $\A\models_s\phi$ and $\A\models_s\psi$; and $\A\models_s\exists x\phi$ if there is $a\in A$ such that $\A\models_{s[a/x]}\phi$. It is obviously required that the free variables of the formula are included in the domain of the assignment.

Next we define what \emph{teams} are, as well as some operations on them that are used to define the semantics of dependence logic.

\begin{definition}
    Let $L$ be a vocabulary.
    \begin{enumerate}
        \item Given an $L$-structure $\A$ and a set $D$ of variables, a \emph{team of $\A$ with domain $D$} is a set of assignments of $\A$ with domain $D$, i.e. a subset of $A^D$. The domain of a team $X$ is denoted by $\dom(X)$.
        \item Given a structure $\A$ and a team $X$ of $\A$, a \emph{supplement function} for $X$ and $\A$ is a function $X\to\Pow(A)\setminus\{\emptyset\}$. Given a supplement function $F$ for $X$ and $\A$ and a variable $x$, we denote by $X[F/x]$ the team $Y$ of $\A$ with domain $\dom(X)\cup\{x\}$ defined by
        \[
            Y = \{ s[a/x] \mid s\in X, a\in F(s) \}.
        \]
        We say that the team has been \emph{supplemented by $F$}.
        \item We denote by $X[A/x]$ the team $Y$ of $\A$ with domain $\dom(X)\cup\{x\}$ defined by
        \[
            Y = \{ s[a/x] \mid a\in A \}.
        \]
        We say that the team has been \emph{duplicated}.
    \end{enumerate}
\end{definition}

Next we define the syntax and semantics of \emph{dependence logic} as it is usually presented in the literature. Another (equivalent) variant is introduced in Section~\ref{An Auxiliary Game} and then used for the rest of the paper.

\begin{definition}
    Let $L$ be a vocabulary.
    \begin{enumerate}
        \item The syntax of $L$-formulas of dependence logic is
        \[
            \phi \Coloneqq \alpha \mid \neg\alpha \mid \dep(\vec x,y) \mid (\phi\land\phi) \mid (\phi\lor\phi) \mid \exists x\phi \mid \forall x\phi,
        \]
        where $\alpha$ is a first-order atomic $L$-formula. We denote dependence logic by $\deplog$.
        \item Given an $L$-formula $\phi\in\deplog$, an $L$-structure $\A$ and a team $X$ of $\A$ whose domain contains the free variables of $\phi$, we define $\A\models_X\phi$  recursively as follows.
        \begin{enumerate}
            \item If $\phi$ is a first-order atomic or negated atomic formula, then $\A\models_X\phi$ if $\A\models_s\phi$ for all $s\in X$.\footnote{Here $\A\models_s\phi(x_0,\dots,x_{n-1})$ means $\A\models\phi(s(x_0),\dots,s(x_{n-1}))$.}
            \item If $\phi$ is $\dep(x_0,\dots,x_{n-1},y)$, then $\A\models_X\phi$ if for all $s,s'\in X$ such that $s(x_i) = s'(x_i)$ for all $i<n$, we have $s(y) = s'(y)$.
            \item If $\phi = \psi\land\theta$, then $\A\models_X\phi$ if $\A\models_X\psi$ and $\A\models_X\theta$.
            \item If $\phi = \psi\lor\theta$, then $\A\models_X\phi$ if there are $Y,Z\subseteq X$ such that $X = Y\cup Z$, $\A\models_Y\psi$ and $\A\models_Z\theta$.
            \item If $\phi = \exists x\psi$, then $\A\models_X\phi$ if there is a supplement function $F$ for $X$ and $\A$ such that $\A\models_{X[F/x]}\psi$.
            \item If $\phi = \forall x\psi$, then $\A\models_X\phi$ if $\A\models_{X[A/x]}\psi$.
        \end{enumerate}
        This is the standard team semantics of dependence logic.
    \end{enumerate}
\end{definition}

The quantifier rank $\qr(\phi)$ of a formula $\phi$ is defined recursively by setting $\qr(\phi) = 0$ for atomic formulas $\phi$, $\qr(\phi\land\psi) = \qr(\phi\lor\psi) = \max\{\qr(\phi), \qr(\psi)\}$ and $\qr(\exists x\phi) = \qr(\forall x\phi) = \qr(\phi)+1$. Note that in~\cite{MR2351449}, the definition is slightly different, as also disjunction increases quantifier rank.

We write $(\A,X)\posEquiv^\deplog_n(\B,Y)$ if $\dom(X) = \dom(Y)$ and for all formulas $\phi$ of dependence logic, with free variables in the domain of $X$ and $Y$, and quantifier rank $\leq n$, we have
\[
    \A\models_X\phi \implies \B\models_Y\phi.
\]
We write $\A\posEquiv^\deplog_n\B$ if $\B$ satisfies all the sentences of quantifier rank $\leq n$ that $\A$ satisfies. If $\A\posEquiv^\deplog_n\B$ for all $n<\omega$, we write $\A\posEquiv^\deplog\B$. Note that $\A\posEquiv^\deplog_n\B$ if and only if $(\A,X)\posEquiv^\deplog_n(\B,Y)$, where $X = Y = \{\emptyset\}$.

\begin{fact}
    Dependence logic is downward closed, meaning that for any $L$-formula $\phi$ of dependence logic, an $L$-structure $\A$ and a team $X$ of $\A$ whose domain contains the free variables of $\phi$, if $\A\models_X\phi$, then $\A\models_Y\phi$ for any $Y\subseteq X$.
\end{fact}

The existential and universal quantifiers do not, at first glance, seem like the semantic duals of each other. However, due to downward closedness, they are, as demonstrated by the following lemma.

\begin{lemma}\label{lemma: alt semantics for universal quantifier}
    Given a formula $\phi$ of dependence logic, $\A\models_X\forall x\phi$ if and only if $\A\models_{X[F/x]}\phi$ for all supplement functions $F$ of $X$ and $\A$.
\end{lemma}
\begin{proof}
    As duplication is a particular kind of supplementation, the direction from right to left is clear. For the other direction, suppose that $\A\models_X\forall x\phi$. Then $\A\models_{X[A/x]}\phi$. Let $F$ be any supplement function of $X$. Note that $X[F/x]\subseteq X[A/x]$. As dependence logic is downward closed, this means that $\A\models_{X[F/x]}\phi$.
\end{proof}

We conclude the section with an important normal form for $\deplog$.

\begin{theorem}[\cite{MR2351449}]\label{theorem: normal form}
    Every sentence of $\deplog$ is equivalent to one of the form
    \[
        \forall x_0\dots\forall x_{m-1}\exists y_0\dots\exists y_{n-1}(\psi\land\theta),
    \]
    where
    \begin{itemize}
        \item $\theta$ is a  quantifier-free first-order formula,
        \item $\psi = \bigwedge_{i<n}\dep(\vec z_i, y_i)$, and
        \item $\vec z_i \in\{x_j \mid j<m\}^{<\omega}$ for all $i<n$.
    \end{itemize}
\end{theorem}

\section{The New Ehrenfeucht--Fraïssé Game}\label{The new game}

We now give an exact definition of the new EF game that we described in the introduction and which we denote by $\EFdep_n(\A,\B)$.

\begin{definition}
    Let $\A$ and $\B$ be $L$-structures with disjoint domains.
    \begin{enumerate}
        \item We denote the (ordinary) Ehrenfeucht--Fraïssé game of length $n$ between $\A$ and $\B$ by $\EF_n(\A,\B)$. The game has two players, $\playerone$ and $\playertwo$, and they take turns picking elements of $A\cup B$. On round $i<n$, either
        \begin{enumerate}
            \item $\playerone$ chooses some $x_i\in A$ and $\playertwo$ responds with $y_i\in B$, or
            \item $\playerone$ choose some $x_i\in B$ and $\playertwo$ responds with $y_i\in A$.
        \end{enumerate}
        Then, let
        \[
            a_i =
            \begin{cases}
                x_i & \text{if $x_i\in A$,} \\
                y_i & \text{otherwise,}
            \end{cases}
            \quad\text{and}\quad
            b_i =
            \begin{cases}
                x_i & \text{if $x_i\in B$,} \\
                y_i & \text{otherwise.}
            \end{cases}
        \]
        $\playertwo$ wins a play
        \[
            (x_0,y_0,\dots,x_{n-1},y_{n-1})
        \]
        if the function $\{(a_i,b_i) \mid i<n\}$ is a partial isomorphism.
        
        If a player $P$ has a winning strategy in the game, we write
        \[
            P\wins\EF_n(\A,\B).
        \]
        \item The game $\EFdep_n(\A,\B)$ of length $n$ between $\A$ and $\B$ is played exactly the same way as $\EF_n(\A,\B)$, with the following addition: on each round $k<n$, in addition to playing an element $x_k\in A\cup B$, $\playerone$ also plays a collection $\W_k\subseteq\Pow(n)$ of sets of indices, as well as two sets
        \[
            \U^+_k,\U^-_k \subseteq \{ R(\vec x) \mid R\in L\cup\{=\}, \vec x\in\{v_0,\dots,v_{n-1}\}^{\ar(R)}, v_k\in\vec x \},
        \]
        of atomic first-order formulas such that $\U^+_k\cap\U^-_k = \emptyset$.
        Player $\playertwo$ wins a play
        \[
            ((x_0,\W_0,\U^+_0,\U^-_k),y_0,\dots,(x_{n-1},\W_{n-1},\U^+_{n-1},\U^-_k),y_{n-1})
        \]
        if the relation $\{ (a_k, b_k) \mid k<n \}$ preserves the atomic formulas of $\bigcup_{k<n}\U^+_k$ and the negations of the atomic formulas of $\bigcup_{k<n}\U^-_k$, i.e. for any $k<n$ and $\phi(v_0,\dots,v_{n-1})\in\U^+_k$,
        \[
            \A\models\phi(a_0,\dots,a_{n-1}) \implies \B\models\phi(b_0,\dots,b_{n-1}),
        \]
        and for any $\phi(v_0,\dots,v_{n-1})\in\U^-_k$,
        \[
            \A\models\neg\phi(a_0,\dots,a_{n-1}) \implies \B\models\neg\phi(b_0,\dots,b_{n-1}).
        \]
    \end{enumerate}
\end{definition}

In the above definition, the set $\W_k$ is the set of dependence commitments made by $\playerone$ on round $k$. If $\{i_0,\dots,i_{k-1}\}\in\W_k$, then $\playerone$ has declared that the move $x_k$ is determined by the moves that were made on rounds $i_0,\dots,i_{k-1}$ in the same model as $x_k$, i.e. by $a_{i_0},\dots,a_{i_{k-1}}$ in the case $x_k\in A$ and by $b_{i_0},\dots,b_{i_{k-1}}$ in the case $x_k\in B$. However, a careful reader may notice that the winning condition for $\playertwo$ does not mention these commitments at all. This is by design: in a single play, a declaration of dependence has no meaning\footnote{Or, rather, the meaning is trivialized, for a singleton team always satisfies all dependence atoms.}. Only when examining multiple plays at once can we determine whether a dependence occurs. This is done below in Definition~\ref{Uniform winning strategy}.

Note that also first-order atomic formulas are treated as commitments of $\playerone$, as illustrated by the sets $\U^+_k$ (atoms to be preserved) and $\U^-_k$ (atoms whose negations are to be preserved). This is due to the fact that in the team semantics of first-order logic, a team satisfies an atomic formula if and only if every assignment separately does. Now, if we think of (the element moves of) a single play of the game as a pair $(s,s')$ of assignments---the moves played in $\A$ making up $s$ and the moves played in $\B$ making up $s'$---then a collection of plays gives rise to a pair of teams. It might be that neither team satisfies, say, $P(v_0)$, but for different reasons: perhaps one assignment $s_0$ in $\A$ does not satisfy $P(v_0)$ while its pair $s'_0$ in $\B$ does, and another assignment $s_1$ satisfies $P(v_0)$ while its pair $s'_1$ does not. Therefore, any atomic formulas $\playerone$ has no intention of satisfying in a uniform manner, should not contribute to $\playertwo$ losing a play in the sense of dependence logic.

Note also that as equations and their negations only need to be preserved when declared so by $\playerone$, the relation $\{ (a_k, b_k) \mid k<n \}$ generally need not be a function, let alone an injection, unlike in the classical game $\EF_n(\A,\B)$.

Next we give an exact definition of what it means for a winning strategy of $\playertwo$ to be uniform.

\begin{definition}\label{Uniform winning strategy}\quad
    \begin{enumerate}
        \item Let
        \[
            p = ((x_0,\W_0,\U^+_0,\U^-_0),y_0,\dots,(x_{n-1},\W_{n-1},\U^+_{n-1},\U^-_{n-1}),y_{n-1})
        \]
        and
        \[
            p' = ((x'_0,\W'_0,(\U^+_0)',(\U^-_0)'),y'_0,\dots,(x'_{n-1},\W'_{n-1},(\U^+_{n-1})',(\U^-_{n-1})'),y'_{n-1})
        \]
        be two plays of $\EFdep_n(\A,\B)$. We say that $\playerone$ has played $p$ and $p'$ \emph{consistently} if for all $k<n$,
        \begin{enumerate}
            \item $\W_k = \W'_k$, $\U^+_k = (\U^+_k)'$ and $\U^-_k = (\U^-_k)'$,
            \item $x_k\in A$ if and only if $x'_k\in A$,
            \item for all atoms $\phi(v_0,\dots,v_{n-1})\in\bigcup_{k<n}\U^+_k$,
            \[
                \A\models\phi(a_0,\dots,a_{n-1}) \quad\text{and}\quad \A\models\phi(a'_0,\dots,a'_{n-1}),
            \]
            \item for all atoms $\phi(v_0,\dots,v_{n-1})\in\bigcup_{k<n}\U^-_k$,
            \[
                \A\models\neg\phi(a_0,\dots,a_{n-1}) \quad\text{and}\quad \A\models\neg\phi(a'_0,\dots,a'_{n-1}),
            \]
            and
            \item for all $W\in\W_k$, whenever $a_i = a'_i$ for all $i\in W$, also $a_k = a'_k$.
        \end{enumerate}
        \item We say that a strategy $\tau$ of $\playertwo$ in $\EFdep_n(\A,\B)$ is \emph{uniformly winning} if it is winning and the following holds: for any two plays
        \[
            ((x_0,\W_0,\U^+_0,\U^-_0),y_0,\dots,(x_{n-1},\W_{n-1},\U^+_{n-1},\U^-_{n-1}),y_{n-1})
        \]
        and
        \[
            ((x'_0,\W'_0,(\U^+_0)',(\U^-_0)'),y'_0,\dots,(x'_{n-1},\W'_{n-1},(\U^+_{n-1})',(\U^-_{n-1})'),y'_{n-1}),
        \]
        where $\playertwo$ has used $\tau$ and $\playerone$ has played consistently, for any $k<n$ and $W\in\W_k$, we have
        \[
            \text{$b_i = b'_i$ for all $i\in W$} \implies b_k = b'_k.
        \]
        When $\playertwo$ has a uniformly winning strategy, we write
        \[
            \playertwo\winsu\EFdep_n(\A,\B).
        \]
    \end{enumerate}
\end{definition}

In the above definition, it is enough to consider only pairs of plays instead of larger sets. This reflects the fact that if a dependence atom is not satisfied by a team, then a subteam of two assignments will already falsify the atom. In Section~\ref{Auxiliary game vs. uniform game}, we show that one can also consider larger sets when convenient.


\section{An Auxiliary Game}\label{An Auxiliary Game}

Let us, for a moment, forget the commitments of $\playerone$ in $\EFdep_n(\A,\B)$ and only consider the element moves. Then, as previously mentioned, a single play
\[
    (x_0,y_0,\dots,x_{n-1},y_{n-1})
\]
can be thought of as a pair of assignments
\[
    (s,s')\in A^D\times B^D,
\]
where $D = \{v_0,\dots,v_{n-1}\}$, such that $a_i = s(v_i)$ and $b_i = s'(v_i)$ for all $i<n$. Furthermore, a set of plays can be seen as a pair of teams $(X,Y)$, where $X$ consists of the assingments $s$ and $Y$ of the assignments $s'$. In such a pair $(X,Y)$, every assignment $s\in X$ clearly corresponds to exactly one $s'\in Y$ and vice versa, so it makes sense to tag each assignment in each team by some index to keep track of the pairings. We now define a notion of a team indexed by a given set to make this formal---and to help us in bookkeeping in later proofs. Giving a team an indexing makes it a multiteam, but team semantics of $\deplog$ for indexed teams will have nothing to do with multiteam semantics of~\cite{MR3833654} but instead collapses back to ordinary team semantics: the truth of ``$\A\models_X\phi$'' remain the same if one forgets about the indexing of $X$ altogether.

Finally, we define an auxiliary game $\EF^\deplog_n((\A,X),(\B,Y))$ based on indexed teams and show that $(\A,X)\posEquiv^\deplog_n(\B,Y)$ implies the existence of a winning strategy for $\playertwo$.

\begin{definition}
    Let $L$ be a vocabulary, $\A$ an $L$-structure, $D$ a set of variables and $I$ a set.
    \begin{enumerate}
        \item An \emph{$I$-indexed team} of $\A$ with domain $D$ is a function $I\to A^D$.
        \item Let $\eta$ be a function whose domain is $I$ and whose values are non-empty sets. A \emph{supplement function} of type $\eta$ for $\A$ is any function $F$ such that $\dom(F) = I$ and $F(i)\in A^{\eta(i)}$ for all $i\in I$.
        \item Let $F$ be a supplement function of type $\eta$ for $\A$. Denote by $I[\eta]$ the set
        \[
            \{ (i,j) \mid i\in I, j\in\eta(i) \}.
        \]
        For a variable $x$ and an $I$-indexed team $X$ of $\A$, we define the supplementation $X[F/x]$ of $X$ by $F$ to be the $I[\eta]$-indexed team $Y$ of $\A$ defined as
        \[
            Y(i,j)(y) =
            \begin{cases}
                F(i)(j) & \text{if $y = x$}, \\
                X(i)(y) & \text{otherwise}.
            \end{cases}
        \]
        In other words, $X[F/x](i,j) = X(i)[F(i)(j)/x]$ for all $(i,j)\in I[\eta]$.
        \item If $X$ is an $I$-indexed team, $J\subseteq I$ and $Y$ is a $J$-indexed team, then we say that $Y$ is a subteam of $X$ if $X\restriction J = Y$.
    \end{enumerate}
\end{definition}

\begin{definition}
    Given an $L$-structure $\A$, an $I$-indexed team $X$ of $\A$ and a formula $\phi$ of dependence logic whose free variables are included in $\dom(X)$, we define $\A\models_X\phi$ recursively as follows.
    \begin{enumerate}
        \item If $\phi$ is first-order atomic or negated atomic, then $\A\models_X\phi$ if for all $i\in I$, $\A\models_{X(i)}\phi$.
        \item If $\phi$ is $\dep(\vec x,y)$, then $\A\models_X\phi$ if for all $i,j\in I$,
        \[
            X(i)(\vec x) = X(j)(\vec x) \implies X(i)(y) = X(j)(y).
        \]
        \item If $\phi = \psi\land\theta$, then $\A\models_X\phi$ if $\A\models_X\psi$ and $\A\models_X\theta$.
        \item If $\phi = \psi\lor\theta$, then $\A\models_X\phi$ if there are $J,K\subseteq I$ such that $J\cup K = I$ and $\A\models_{X\restriction J}\psi$ and $\A\models_{X\restriction K}\theta$.
        \item If $\phi = \exists x \psi$, then $\A\models_X\phi$ if there is a supplement function $F$ for $\A$ such that $\A\models_{X[F/x]}\psi$.
        \item If $\phi = \forall x\psi$, then $\A\models_X\phi$ if $\A\models_{X[F/x]}\psi$ for all supplement functions $F$ for $\A$.
    \end{enumerate}
\end{definition}

Next we prove that the indexing of a team by some set $I$, possibly making the team a multiset of assignments, does not in any way affect the semantics as we have defined it.

\begin{lemma}\label{lemma: indexing does not change semantics}
    Given an $I$-indexed team $X$, denote by $X^*$ the (ordinary) team
    \[
        \{ X(i) \mid i\in I \}.
    \]
    Then for any formula $\phi$ of dependence logic,
    \[
        \A\models_X\phi \iff \A\models_{X^*}\phi.
    \]
\end{lemma}
\begin{proof}
    We proceed by induction on $\phi$. The case for atomic formulas and connectives is clear, so we only consider the quantifiers. The universal quantifier case is the dual of the existential quantifier case, by Lemma~\ref{lemma: alt semantics for universal quantifier}. Suppose that $\phi = \exists x\psi$ and $\A\models_X\phi$. Then there is a supplement function $F$ such that $\A\models_{X[F/x]}\psi$. Let $\eta$ be the type of $F$. We define an (ordinary) supplement function $G$ on $X^*$ by setting
    \[
        G(X(i)) = \{ F(i)(j) \mid j\in\eta(i) \}.
    \]
    By the induction hypothesis, we have $\A\models_{X[F/x]^*}\psi$, so it is enought to show that $X[F/x]^* = X^*[G/x]$.
    
    First suppose that $s\in X[F/x]^*$. Then $s = X[F/x](i,j)$ for some $(i,j)\in I[\eta]$. This means that $i\in I$ and $j\in\eta(i)$. Now $s = X(i)[F(i)(j)/x]$. As $j\in\eta(i)$, we have $F(i)(j)\in G(X(i))$, and by definition, $X(i)\in X^*$, so it follows that $s\in X^*[G/x]$.

    Then suppose that $s\in X^*[G/x]$. Then $s = s'[a/x]$ for some $s'\in X^*$ and $a\in G(s')$. Now there is some $i\in I$ such that $s' = X(i)$. As $a\in G(X(i))$, there is $j\in\eta(i)$ such that $a = F(i)(j)$. Thus $s = X(i)[F(i)(j)/x] = X[F/x](i,j)$, whence $s\in X[F/x]^*$.

    Conversely, suppose that $\A\models_{X^*}\phi$. Then there is a(n ordinary) supplement function $G$ on $X^*$ such that $\A\models_{X^*[G/x]}\psi$. Let $\eta$ be the function
    \[
        i\mapsto G(X(i)), i\in I.
    \]
    Define a supplement function $F$ of type $\eta$ by setting
    \[
        F(i)(a) = a.
    \]
    Again, it is sufficient to show that $X[F/x]^* = X^*[G/x]$.

    First suppose that $s\in X[F/x]^*$. Then $s = X[F/x](i,a)$ for some $(i,a)\in I[\eta]$. Now $s = X(i)[F(i)(a)/x]$. By definition, $F(i)(a) = a\in G(X(i))$ and $X(i)\in X^*$, so $s\in X^*[G/x]$.

    Then suppose that $s\in X^*[G/x]$. Then $s = s'[a/x]$ for some $s'\in X^*$ and $a\in G(s')$. Again, $s' = X(i)$ and $a\in G(X(i))$ for some $i\in I$. As $a = F(i)(a)$, we have $s = X(i)[F(i)(a)/x] = X[F/x](i,a)$, and hence $s\in X[F/x]^*$.
\end{proof}

\begin{lemma}
    For any $I$-indexed team $X$ of $\A$ and formula $\phi$ of $\deplog$, we have $\A\models_X\forall x\phi$ if and only if $\A\models_{X[F/x]}\phi$, where $F$ is any duplicating supplement function, i.e. for each $i\in I[\eta]$, where $\eta$ is the type of $F$, $F(i)$ lists all elements of $A$ (possibly with repetition).
\end{lemma}
\begin{proof}
    Combine Lemmas~\ref{lemma: alt semantics for universal quantifier} and~\ref{lemma: indexing does not change semantics}.
\end{proof}

Next we define an auxiliary game that we use to prove that $\EFdep_n(\A,\B)$ captures equivalence in $\deplog$. It is an indexed variant of the EF game introduced in~\cite{MR2351449} for $\deplog$, without the disjunction moves.
Additionally, positions in the game of~\cite{MR2351449} are teams, but in the game $\EFdep_n(\A,\B)$, moves are supplement functions. This way, moves of the game are the same kind of objects as the ones used to define the semantics of quantifiers, making the game more analogous to the classic EF game.

\begin{definition}
    Let $\A$ and $\B$ be $L$-structures, $I$ a set, $D = \{v_k \mid k<\alpha\}$ for some ordinal $\alpha$, and $X$ and $Y$ $I$-indexed teams of $\A$ and $\B$, respectively, with domain $D$. The game $\EF^\deplog_n((\A,X),(\B,Y))$ has two players, $\playerone$ and $\playertwo$, and $n$ rounds. Denote $I_0 = I$, $X_0 = X$ and $Y_0 = Y$. On round $k$, $\playerone$ chooses a supplement function $F_k$ for either $\A$ or $\B$, of some type $\eta_k$ such that $\dom(\eta_k) = I_k$. Then $\playertwo$ chooses a supplement function $G_k$ for $\B$ if $F_k$ was for $\A$ and for $\B$ otherwise, so that the type of $G_k$ is also $\eta_k$. Then $I_{k+1} \coloneqq I_k[\eta_k]$, $X_{k+1} = X_k[F_k/v_{\alpha+k}]$ and $Y_{k+1} = Y_k[G_k/v_{\alpha+k}]$. Denote by $H_k^\A$ the move that was for $\A$ and by $H_k^\B$ the one that was for $\B$. After all $n$ rounds have been played, look at the $I_n$-indexed teams
    \begin{align*}
        X_n &= X[H^\A_0/v_\alpha]\dots[H^\A_{n-1}/v_{\alpha+n-1}] \quad\text{and} \\
        Y_n &= Y[H^\B_0/v_\alpha]\dots[H^\B_{n-1}/v_{\alpha+n-1}].
    \end{align*}
    Then $\playertwo$ wins if for all atomic and negated atomic formulas $\phi$ (including dependence atoms), with free variables in $\{v_k \mid k<\alpha+n\}$, we have
    \[
        \A\models_{X_n}\phi \implies \B\models_{Y_n}\phi.
    \]
\end{definition}

\begin{lemma}\label{lemma: from logic to auxiliary game}
    For all $I$ and $I$-indexed teams $X$ and $Y$ of $\A$ and $\B$, respectively, with domain $D_\alpha$ for some $\alpha$, if $(\A,X)\posEquiv^\deplog_n(\B,Y)$, then
    \[
        \playertwo\wins\EF^\deplog_n((\A,X),(\B,Y)).
    \]
\end{lemma}
\begin{proof}
    We prove the claim by induction on $n$. First, fix $X$ and $Y$ and suppose that $(\A,X)\posEquiv^\deplog_0(\B,Y)$. This means that for all atomic and negated atomic formulas $\phi$,
    \[
        \A\models_X\phi \implies \B\models_Y\phi.
    \]
    As this is the winning condition of $\playertwo$ in $\EF^\deplog_0((\A,X),(\B,Y))$, the only strategy of $\playertwo$ in this $0$-round game is winning.

    Then assume, as the induction hypothesis, that for any $X$ and $Y$, if $(\A,X)\posEquiv^\deplog_n(\B,Y)$, then
    \[
        \playertwo\wins\EF^\deplog_n((\A,X),(\B,Y)).
    \]
    Now fix $X$ and $Y$ and suppose that $(\A,X)\posEquiv^\deplog_{n+1}(\B,Y)$. Let us play $\EF^\deplog_{n+1}((\A,X),(\B,Y))$. On the first round, $\playerone$ plays a supplement function $F_0$. Suppose that $F_0 = H^\A_0$ (the case $F_0 = H^\B_0$ is symmetric). Now, let
    \[
        \Phi = \bigwedge\{\phi(\vec x) \mid \vec x\in D_{\alpha+1}^{<\omega}, \qr(\phi)\leq n, \A\models_{X[F_0/v_\alpha]}\phi \}.
    \]
    Clearly $\A\models_X\exists v_\alpha\Phi$ and $\qr(\exists v_\alpha\Phi) \leq n+1$. Thus, $\B\models_Y\exists v_\alpha\Phi$, whence there is some $G_0$ such that $\B\models_{Y[G_0/v_\alpha]}\Phi$. By downward closedness, we may assume that the type of $G_0$ is the same as the type of $F_0$. By the definition of $\Phi$, we have $\B\models_{Y[G_0/v_\alpha]}\phi$ for all $\phi$ of quantifier rank $\leq n$ such that $\A\models_{X[F_0/v_\alpha]}\phi$. Hence
    \[
        (A,X[F_0/v_\alpha]) \posEquiv^\deplog_n (\B,Y[G_0/v_\alpha]),
    \]
    which, by the induction hypothesis, means that
    \[
        \playertwo\wins\EF^\deplog_n((A,X[F_0/v_\alpha]),(\B,Y[G_0/v_\alpha])).
    \]
    Hence, in $\EF^\deplog_{n+1}((\A,X),(\B,Y))$, if $\playertwo$ plays $G_0$ as a response to $\playerone$ playing $F_0$ on the first round, she has a winning strategy from that round onwards. Hence she wins the play.
\end{proof}

We could also prove a converse for the previous result, but that is not without complications, as we have dropped the disjunction move that was included in the game of~\cite{MR2351449}. If we redefine quantifier rank to increase with disjunction and add the disjunction move to the game, then we obviously recover the results of~\cite{MR2351449}.

\section{The Auxiliary Game vs. the Uniform Game}\label{Auxiliary game vs. uniform game}

In this section, we connect the auxiliary game defined in Section~\ref{An Auxiliary Game} to our new game of Section~\ref{The new game}. We will then leverage this connection in Section~\ref{Uniform game vs. logic} to prove our main result.

The following lemma formulates the consistency and uniformity conditions of $\EFdep$ in terms of indexed teams.

\begin{lemma}\label{lemma: characterisation of the uniform game}\quad
    \begin{enumerate}
        \item Suppose that $I$ is a set and
        \[
            P = \{ ((x_0^i,\W_0,\U^+_0,\U^-_0),y_0^i,\dots,(x_{n-1}^i,\W_{n-1},\U^+_{n-1},\U^-_{n-1}),y_{n-1}^i) \mid i\in I \}
        \]
        is a set of plays of $\EFdep_n(\A,\B)$. Denote
        \[
            a_k^i =
            \begin{cases}
                x_k^i & \text{if $x_k\in A$}, \\
                y_k^i & \text{otherwise},
            \end{cases}
        \]
        and let $X$ be the $I$-indexed team
        \[
            i\mapsto\{(v_k,a^i_k) \mid k<n\}, \quad i\in I
        \]        
        of $\A$. Then $\playerone$ has played the plays of $P$ (pairwise) consistently if and only if
        \begin{enumerate}
            \item for all $i,j\in I$ and $k<n$, $x_k^i\in A$ if and only if $x_k^j\in A$,
            \item for all $k<n$ and $\phi\in\U^+_k$, $\A\models_X\phi$,
            \item for all $k<n$ and $\phi\in\U^-_k$, $\A\models_X\neg\phi$,
            and
            \item for all $k<n$ and $W\in\W_k$,
            \[
                \A\models_X\dep(\vec x, v_k),
            \]
            where $\vec x$ enumerates $\{v_l \mid l\in W\}$.
        \end{enumerate}

        \item A strategy $\tau$ of $\playertwo$ in $\EFdep_n(\A,\B)$ is uniformly winning if and only if the following holds. Suppose that $I$ is a set and
        \[
            \{ ((x_0^i,\W_0,\U^+_0,\U^-_0),y_0^i,\dots,(x_{n-1}^i,\W_{n-1},\U^+_{n-1},\U^-_{n-1}),y_{n-1}^i) \mid i\in I \}
        \]
        is a set of plays of $\EFdep_n(\A,\B)$ such that $\playertwo$ has used $\tau$ and $\playerone$ has played consistently. Denote
        \[
            a_k^i =
            \begin{cases}
                x_k^i & \text{if $x_k\in A$}, \\
                y_k^i & \text{otherwise},
            \end{cases}
            \quad\text{and}\quad
            b_k^i =
            \begin{cases}
                x_k^i & \text{if $x_k\in B$}, \\
                y_k^i & \text{otherwise}.
            \end{cases}
        \]
        Let $X$ be the $I$-indexed team
        \[
            i\mapsto\{(v_k,a^i_k) \mid k<n\}, \quad i\in I,
        \]
        of $\A$ and $Y$ the $I$-indexed team
        \[
            i\mapsto\{(v_k,b^i_k) \mid k<n\}, \quad i\in I,
        \]
        of $\B$. Then,
        \begin{enumerate}
            \item for all $k<n$ and $\phi\in\U^+_k$, $\B\models_Y\phi$,
            \item for all $k<n$ and $\phi\in\U^-_k$, $\B\models_Y\neg\phi$, and
            \item for all $k<n$ and $W\in\W_k$,
            \[
                \B\models_Y\dep(\vec x,v_k),
            \]
            where $\vec x$ enumerates $\{v_l \mid l\in W\}$.
        \end{enumerate}
    \end{enumerate}
\end{lemma}
\begin{proof}\quad
    \begin{enumerate}
        \item We show that $\playerone$ has played $P$ pairwise consistently if and only if the conditions of the statement of the lemma hold.
        \begin{enumerate}
            \item This is the same condition as in the definition of $\playerone$ playing consistently.
            \item Fix $k<n$. By flatness and the fact that $a^i_l = X(i)(v_l)$ for all $l<n$, for any $\phi(v_0,\dots,v_{n-1})\in\U^+_k$, we have
            \[
                \text{$\A\models\phi(a_0^i,\dots,a_{n-1}^i)$ for all $i\in I$} \iff \A\models_X\phi.
            \]
            \item The case for $\U^-_k$ is similar to the previous item.
            \item Fix $k<n$ and $W\in\W_k$ and let $\vec x$ enumerate $\{v_l \mid l\in W\}$. As $X(i)(v_l) = a^i_l$ for all $l<n$, we have
            \begin{align*}
                &\forall i,j\in I \bigg[ \forall l\in W (a^i_l = a^j_l) \implies a^i_k = a^j_k \bigg] \\
                &\iff \forall i,j\in I \bigg[ \forall l\in W (X(i)(v_l) = X(j)(v_l)) \\
                &\hspace{8em}\implies X(i)(v_k) = X(j)(v_k) \bigg] \\
                &\iff \A\models_X\dep(\vec x, v_k).
            \end{align*}
        \end{enumerate}

        \item Let $\tau$ be uniformly winning. Let $I$, $x^i_l$, $\W_l$, $\U^+_l$, $\U^-_l$, $y^i_l$, $X$ and $Y$ be as in the statement of the lemma and that $\playerone$ has played consistently. Thus, as $\tau$ wins uniformly, we have
        \[
            \text{$b^i_l = b^j_l$ for all $l\in W$} \implies b^i_k = b^j_k
        \]
        for all $i,j\in I$, $k<n$ and $W\in\W_k$. As we have $b^i_j = Y(i)(v_j)$ for all $i\in I$ and $j<n$, we immediately obtain $\B\models_Y\dep(\vec x_W,v_k)$ for each $W\in\W_k$. Fix $\phi\in\U^+_k$. As $\tau$ is winning, we have
        \[
            \A\models\phi(a^i_0,\dots,a^i_{n-1}) \implies \B\models\phi(b^i_0,\dots,b^i_{n-1})
        \]
        for all $i\in I$. As $\playerone$ has played consistently, we have $\A\models\phi(a^i_0,\dots,a^i_{n-1})$. Thus we obtain $\B\models\phi(b^i_0,\dots,b^i_{n-1})$. By flatness, $\B\models_Y\phi$. The case for $\phi\in\U^-_k$ is similar.
    
        For the converse, suppose that for any $I$, $x^i_l$, $\W_l$, $\U^+_l$, $\U^-_l$, $y^i_l$, $X$ and $Y$ that are as above, the conditions of the statement of the lemma hold. We show that $\tau$ is uniformly winning. To first show that $\tau$ is winning, let 
        \[
            ((x^0_0,\W_0,\U^+_0,\U^-_0),y^0_0,\dots,(x^0_{n-1},\W_{n-1},\U^+_{n-1},\U^-_{n-1}),y^0_{n-1})
        \]
        be a play where $\playertwo$ has used $\tau$. Now, fix $k<n$ and $\phi(v_0,\dots,v_{n-1})\in\U^+_k$ and suppose that $\A\models\phi(a_0,\dots,a_{n-1})$. Let $I = \{0\}$ and define $x^0_l = x_l$ and $y^0_l = y_l$ and let $X$ and $Y$ be the $I$-indexed teams made from the singleton set
        \[
            \{ ((x^0_0,\W_0,\U^+_0,\U^-_0),y^0_0,\dots,(x^0_{n-1},\W_{n-1},\U^+_{n-1},\U^-_{n-1}),y^0_{n-1}) \}
        \]
        of plays. Clearly $\playerone$ has played the set of plays pairwise consistently. By applying the statement of the lemma, we obtain $\B\models_Y\phi$, which by flatness means $\B\models\phi(b_0,\dots,b_{n-1})$. The treatment of any $\phi\in\U^-_k$ is similar.
        
        In order to show that $\tau$ wins uniformly, let
        \[
            ((x_0,\W_0,\U^+_0,\U^-_0),y_0,\dots,(x_{n-1},\W_{n-1},\U^+_{n-1},\U^-_{n-1}),y_{n-1})
        \]
        and
        \[
            ((\hat x_0,\hat\W_0,\hat\U^+_0,\hat\U^-_0),\hat y_0,\dots,(\hat x_{n-1},\hat\W_{n-1},\hat\U^+_{n-1},\hat\U^-_{n-1}),\hat y_{n-1})
        \]
        be two plays where $\playertwo$ has used $\tau$ and $\playerone$ has played consistently. It follows that $\W_l = \hat\W_l$, $\U^+_l = \hat\U^+_l$ and $\U^-_l = \hat\U^-_l$ for all $l<n$. Fix $k<n$ and $W\in\W_k$, and suppose that $b_l = \hat b_j$ for all $l\in W$. Now, let $I = \{0,1\}$ and define $x^0_l = x_l$, $y^0_l = y_l$, $x^1_l = \hat x_l$ and $y^1_l = \hat y_j$. Let $X$ and $Y$ consist of the move sequences of the two players as before. As $\playerone$ has played consistently, we have $\A\models_X\dep(\vec x_W,v_k)$ for all $W\in\W_k$. We may apply our initial assumption to obtain $\B\models_Y\dep(\vec x_W,v_k)$ for all $W\in\W_k$. This means that for all $W\in\W_k$,
        \[
            \text{$Y(0)(v_l) = Y(1)(v_l)$ for all $l\in W$} \implies Y(0)(v_k) = Y(1)(v_k).
        \]
        As $Y(0)(v_l) = b^0_l = b_l = \hat b_l = b^1_l = Y(1)(v_l)$ for all $l\in W$, we must also have $b_k = b^0_k = Y(0)(v_k) = Y(1)(v_k) = b^1_k = \hat b_k$. This concludes the proof. \qedhere
    \end{enumerate}
\end{proof}

\begin{lemma}\label{lemma: from auxiliary game to uniform game}
    Denote by $N$ the $1$-indexed team $0\mapsto\emptyset$, and let $n<\omega$. If
    \[
        \playertwo\wins\EF_n^\deplog((\A,N),(\B,N)),
    \]
    then
    \[
        \playertwo\winsu\EFdep_n(\A,\B).
    \]
\end{lemma}
\begin{proof}
    Let $\tau$ be a strategy of $\playertwo$ in $\EF_n^\deplog((\A,N),(\B,N))$. We define a uniform strategy of $\playertwo$ in $\EFdep_n(\A,\B)$ in pieces for each way that $\playerone$ can play consistently. For this, fix $(i_0,\dots,i_{n-1})\in 2^n$ and, for each $k<n$, $\W_k\subseteq\Pow(n)$ and $\U^+_k,\U^-_k\subseteq\{ R(\vec x) \mid R\in L, \vec x\in\{v_0,\dots,v_{n-1}\}^{<\omega}, v_k\in\vec x\}$. Let $P$ be the set of all plays of $\EFdep_n(\A,\B)$ where $\playerone$ has played consistently and, on each round $k<n$, $\playerone$ has played the commitments $\W_k$, $\U^+_k$ and $\U^-_k$, and the element $x_k$ is in $\A$ if and only if $i_k = 0$. Now, consider the set of all possible moves $x_0^i$, $i<\alpha_0$, such that the move can lead to a play in $P$. Let $X_0 = Y_0 = N$ and $I_0 = 1$. Denote by $\eta_0$ the function with domain $1$ and $\eta_0(0) = \alpha_0$. We then define a supplement function $F_0$ of type $\eta_0$ such that $\dom(F_0) = 1$, $F_0(0) = (x_0^i)_{i<\alpha_0}$. Now, $\tau$ produces a supplement function $G_0$ of type $\eta_0$ such that $\playertwo\wins\EF_{n-1}^\deplog((\A,X_1),(\B,Y_1))$, where $I_1 = I_0[\eta_0]$ and $X_1$ and $Y_1$ are the $I_1$-indexed teams $X_0[H_0^\A/v_0]$ and $Y_0[H_0^\B/v_0]$, respectively.

    Then, for each $i\in I_1$, consider all the moves $x_1^{(i,j)}$, $j<\alpha_1$, that $\playerone$ can play so that the move sequence $(x_0^i,G_0(0)(i),x_1^{(i,j)})$ can lead to a play in $P$. Let $\eta_1$ have domain $I_1$ and $\eta_1(i) = \alpha^i_1$ for all $i\in I_1$. We define a supplement function $F_1$ of type $\eta_1$ by setting $F_1(i) = (x_1^{i,j})_{j<\alpha_1^i}$. Then $\tau$ produces $G_1$ of type $\eta_1$ such that $\playertwo\wins\EF^\deplog_{n-1}((\A,X_2),(\B,Y_2))$, where $I_2 = I_1[\eta_1]$ and $X_2$ and $Y_2$ are the $I_2$-indexed teas $X_1[H_1^\A/v_1]$ and $Y_1[H_1^\B/V_1]$.

    By continuing this way, we accumulate teams $X_n$ and $Y_n$ that consist of all possible movesets of $\playerone$ where he has played consistently and according to the specification $(i_0,\dots,i_{n-1})$, while playing $\W_k$, $\U^+_k$ and $\U^-_k$ on round $k$, and the responses of $\playertwo$ as described. As $\playerone$ has played consistently, it so happens that $\A\models_{X_n}R(\vec x)$ for all $R(\vec x)\in\bigcup_{k<n}\U^+_k$ and $\A\models_{X_n}\neg R(\vec x)$ for all $R(\vec x)\in\bigcup_{k<n}\U^-_k$ and $\A\models_{X_n}\dep(\vec x,v_k)$ for all $W\in\W_k$, where $\vec x$ enumerates $\{v_l \mid l\in W\}$. As $\tau$ is winning in $\EF_n^\deplog((\A,N),(\B,N))$, we have
    \[
        \A\models_{X_n}\phi \implies \B\models_{Y_n}\phi
    \]
    for all atomic $\phi$. Hence, the strategy where $\playertwo$ plays $G_k(i)$ on round $k$ of $\EFdep_n(\A,\B)$ as a response to $\playerone$ playing $F_k(i)$ is uniformly winning.
\end{proof}

We could prove the converse of Lemma~\ref{lemma: from auxiliary game to uniform game} with a few redefinitions. In $\EF_n^\deplog((\A,X),(\B,Y))$ the obvious strategy of $\playertwo$ would be to consider assignments of each index $i\in I$ in $X$ as moves played in $\A$ and in $Y$ as moves played in $\B$ in a single play of $\EFdep_n(\A,\B)$ and build supplement functions out of possible responses in the uniform game to moves by $\playerone$ given by the supplement function. However, we would need to consider a variant of $\deplog$ where we require that dependence atoms only occur in a form $\dep(\vec x, v_k)$, where $\vec x = (v_{i_0},\dots,v_{i_{m-1}})$ and $k>i_0,\dots,i_{m-1}$. We would also need to modify $\EFdep_n(\A,\B)$ so that $\W_i\subseteq\Pow(i)$, i.e. $\playerone$ does not make promises that his move is determined by future moves (which is possible with our current definition).

\section{The Main Result}\label{Uniform game vs. logic}

We are ready to prove the main result of this paper:

\begin{theorem}
    The following are equivalent.
    \begin{enumerate}
        \item\label{item: II wins uniformly for all n}
        $\playertwo\winsu\EFdep_n(\A,\B)$ for all $n<\omega$.
        \item\label{item: Dependence logic is preserved}
        $\A\posEquiv^\deplog\B$.
    \end{enumerate}
\end{theorem}
\begin{proof}
    \ref{item: Dependence logic is preserved}$\implies$\ref{item: II wins uniformly for all n}: Suppose that $\A\posEquiv^\deplog\B$. By Lemma~\ref{lemma: from logic to auxiliary game},
    \[
        \playertwo\wins\EF^\deplog_n((\A,N),(\B,N))
    \]
    for all $n<\omega$, whence by Lemma~\ref{lemma: from auxiliary game to uniform game}, we have $\playertwo\winsu\EFdep_n(\A,\B)$ for all $n<\omega$.

    \ref{item: II wins uniformly for all n}$\implies$\ref{item: Dependence logic is preserved}: Suppose that $\playertwo\winsu\EFdep_n(\A,\B)$ for all $n<\omega$. Let $\phi$ be an arbitrary sentence of dependence logic. We may assume that $\phi$ is
    \[
         \forall v_0\dots\forall v_{m-1}\exists v_m\dots\exists v_{m+k-1}(\psi\land\theta),
    \]
    where $\theta$ is a  quantifier-free first-order formula and $\psi$ is
    \[
        \bigwedge_{i<k}\dep(\vec u_i, v_{m+i})
    \]
    with $\vec u_i \in\{v_j \mid j<m\}^{<\omega}$ for all $i<k$. Furthermore, we can assume that $\theta$ is in disjunctive normal form and let $\theta$ be
    \[
        \bigvee_{i<l}\left( \bigwedge_{j\in U^+_i}\theta^+_j \land \bigwedge_{j\in U^-_i}\neg\theta^-_j \right).
    \]
    Then fix a uniform winning strategy $\tau$ for $\playertwo$ in the game of length $n = m+k$. We now aim at showing that
    \[
        \A\models_N\phi \implies \B\models_N\phi,
    \]
    so suppose that $\A\models_N\phi$. Now, let $X_0 = Y_0 = N$, let $F^*_0,\dots,F^*_{m-1}$ and $G_0,\dots,G_{m-1}$ be supplement functions such that $F^*_i$ duplicates the team $X^*_i\coloneqq N[F_0/v_0]\dots[F_{i-1}/v_{i-1}]$, $G_i$ duplicates $Y_i\coloneqq N[G_0/v_0]\dots[G_{i-1}/v_{i-1}]$, $F^*_i$ and $G_i$ have the same type, and for every $b\in B$,
    \[
        \{F^*_i(j) \mid G_i(j) = b \} = A
    \]
    (i.e. every element of $\A$ occurs together with all elements of $\B$). Now,
    \[
        \A\models_{X^*_m}\exists v_m\dots\exists v_{m+k-1}(\psi\land\theta),
    \]
    so there are supplement functions $F_m,\dots,F_{m+k-1}$ such that
    \[
        \A\models_{X^*_m[F_m/v_m]\dots[F_{m+k-1}/v_{m+k-1}]}\psi\land\theta.
    \]
    Then let
    \begin{align*}
        X^* &= X^*_m[F_m/v_m]\dots[F_{m+k-1}/v_{m+k-1}] \\
        &= N[F^*_0/v_0]\dots[F^*_{m-1}/v_{m-1}][F_m/v_m]\dots[F_{m+k-1}/v_{m+k-1}]
    \end{align*}
    and let $I$ be its index set.
    Now there is $i^*<l$ such that
    \[
        \A\models_{X^*} \bigwedge_{j\in U^+_{i^*}}\theta^+_j \land \bigwedge_{j\in U^-_{i^*}}\neg\theta^-_j.
    \]
    Let $\U^+_p$ consist of all the $\theta_j^+$ such that $j\in U^+_{i^*}$, $\U^-_p$ of all the $\theta_j^-$ such that $j\in U^-_{i^*}$ and $\W_p = \{v_j \mid v_j\in\vec u_p\}$ when $m\leq p < m+k$ and $\W_p = \emptyset$ otherwise.

    Now, let us play $\EFdep_{m+k}(\A,\B)$ $|I|$-many times. In the $i^{\text{th}}$ game, on round $j<m$, we let $\playerone$ play $(Y(i)(j),\W_j,\U^+_j,\U^-_j)$ and denote by $a^i_j$ the response of $\playertwo$ produced by her uniform winning strategy. Now, we define supplement functions $F_i$ so that $F_i(j) = a^i_j$. Then let
    \[
        X = N[F_0/v_0]\dots[F_{m+k-1}/v_{m+k-1}].
    \]
    Now, $\{X(i) \mid i\in I\}\subseteq\{X^*(i) \mid i\in I\}$, so $X$ also satisfies $\psi$ and 
    \[
        \bigwedge_{j\in U^+_{i^*}}\theta^+_j \land \bigwedge_{j\in U^-_{i^*}}\neg\theta^-_j
    \]
    by downward closedness and Lemma~\ref{lemma: indexing does not change semantics}. On round $j\geq m$, we let $\playerone$ play
    $(X(i)(j),\W_j,\U^+_j,\U^-_j)$ and denote by $b^i_j$ the response of $\playertwo$. Now, consider the set of all these plays. Clearly $\playerone$ has played consistently, as demonstrated by the fact that
    \[
        \A\models_X\bigwedge_{i<k}\dep(\vec u_i,v_{m+i})\land\bigwedge_{j\in U^+_{i^*}}\theta^+_j \land \bigwedge_{j\in U^-_{i^*}}\neg\theta^-_j.
    \]
    Then denote by $Y$ the team $Y_m[G_m/v_m]\dots[G_{m+k-1}/v_{m+k-1}]$, where $G_{m+j}$ are defined so that $Y(i)(j) = b^i_j$. As $\playertwo$ wins uniformly, we have
    \[
        \B\models_Y\dep(\vec u_p,v_{m+p})\land\bigwedge_{j\in U^+_{i^*}}\theta^+_j \land \bigwedge_{j\in U^-_{i^*}}\neg\theta^-_j
    \]
    for all $p<m+k$. Thus, $\B\models_Y\psi\land\theta$, whence $\B\models_N\phi$, as desired.
\end{proof}

\section{Future Work}

Given that Ehrenfeucht--Fraïssé type games are often used to prove inexpressivity results, it would be illuminating to have more examples. For instance, could we use the game to show that finiteness is not definable in $\deplog$? Even more interesting would be to be able to show that a property not known to be undefinable in $\deplog$ is not definable.

Having to settle for the use of a normal form in the proof of our main result is somewhat unsatisfying, and we also lose the connection between the length of the game and the quantifier rank of formulas that are preserved. The alternative could be to restrict both the syntax of $\deplog$ and the commitments made in the game (and this would also reconnect quantifier rank and game length), but this also seems unsatisfactory. Is it possible to connect the length of the game to quantifier rank of the formulas in some reasonable manner?

Lemma~\ref{lemma: from auxiliary game to uniform game} utilizes the fact that $\EF^{\deplog}_n((\A,X),(\B,Y))$ is symmetric, which is made possible by Lemma~\ref{lemma: alt semantics for universal quantifier}---a result that does not necessarily hold if one moves from $\deplog$ to some other logic that is not downward closed. The main result also uses some amount of downward closedness. Therefore the results of this paper cannot immediately be lifted to, say, inclusion logic. Thus, a question arises: can a similar game be found for other team logics, most notably inclusion logic and independence logic?

\bibliographystyle{plain}
\bibliography{efgame}
\end{document}